\documentclass[11pt]{amsart}
\usepackage{amssymb,amsmath,amsthm}
\usepackage[all]{xy}
\xyoption{arc}
\usepackage[dvips]{graphicx}

\usepackage{a4wide}

\newcommand{\alg}{\mathrm{alg}}
\newcommand{\CH}{\operatorname{CH}}
\newcommand{\chern}{\operatorname{ch}}

\newcommand{\e}{\operatorname{e}}

\newcommand{\Exp}{\operatorname{E}}

\newcommand{\Fourier}{\mathcal{F}}
\newcommand{\Hom}{\operatorname{Hom}}
\newcommand{\id}{\mathrm{id}}

\newcommand{\Ker}{\operatorname{Ker}}
\newcommand{\Mot}{\mathcal{M}}
\newcommand{\Pic}{\operatorname{Pic}}
\newcommand{\pr}{\mathrm{pr}}
\newcommand{\pprime}{{\prime\prime}}
\newcommand{\pt}{\mathrm{pt}}
\newcommand{\Sp}{\operatorname{Sp}}
\newcommand{\Spec}{\operatorname{Spec}}
\newcommand{\Sym}{\operatorname{Sym}}

\newcommand{\CC}{C}
\newcommand{\FF}{\mathcal{F}}
\newcommand{\GG}{\mathcal{G}}
\newcommand{\OO}{\mathcal{O}}
\newcommand{\PP}{\mathcal{P}}

\newcommand{\ag}{\mathfrak{a}}

\renewcommand{\P}{\mathbb{P}}
\newcommand{\Q}{\mathbb{Q}}
\newcommand{\mQ}{\mathbb{Q}}

\newcommand{\Z}{\mathbb{Z}}

\newcommand{\Cbul}{\CC^{[\bullet]}}

\newcommand{\sub}{\subset}

\newcommand{\ot}{\otimes}

\newcommand{\kbar}{\bar{k}}

\renewcommand{\a}{\alpha}

\newcommand{\De}{\Delta}
\newcommand{\eps}{\epsilon}

\renewcommand{\th}{\theta}

\numberwithin{equation}{subsection}
\newtheorem{thm}[subsection]{Theorem}
\newtheorem{prop}[subsection]{Proposition}
\newtheorem{lem}[subsection]{Lemma}
\newtheorem{cor}[subsection]{Corollary}

\theoremstyle{definition}

\newtheorem{defi}[subsection]{Definition}

\newtheorem{rem}[subsection]{Remark}

\newdimen\normalparindent       
\parindent=\normalparindent

\let\geq=\geqslant
\let\ge=\geqslant
\let\leq=\leqslant
\let\le=\leqslant
\begin{document}

\title{Divided powers in Chow rings and integral Fourier transforms}

\author{Ben Moonen}
\address{Department of Mathematics, University of Amsterdam, P.O.\ Box 94248, 1090~GE Amsterdam, The Netherlands}
\email{b.j.j.moonen@uva.nl}

\author{Alexander Polishchuk}
\address{Department of Mathematics, University of Oregon, Eugene, OR 97403, USA}
\email{apolish@uoregon.edu}
\thanks{The research of Alexander Polishchuk was partially supported by NSF grant DMS-0601034}


\begin{abstract} 
We prove that for any monoid scheme~$M$ over a field with proper multiplication maps $M \times M \to M$, we have a natural PD-structure on the ideal $\CH_{>0}(M) \subset \CH_*(M)$ with regard to the Pontryagin ring structure. Further we investigate to what extent it is possible to define a Fourier transform on the motive with integral coefficients of the Jacobian of a curve. For a hyperelliptic curve of genus~$g$ with sufficiently many $k$-rational Weierstrass points, we construct such an integral Fourier transform with all the usual properties up to $2^N$-torsion, where $N = 1 + \lfloor \log_2(3g)\rfloor$. As a consequence we obtain, over $k=\kbar$, a PD-structure (for the intersection product) on $2^N \cdot \ag$, where $\ag \subset \CH(J)$ is the augmentation ideal. We show that a factor~$2$ in the properties
of an integral Fourier transform cannot be eliminated even for elliptic curves over an algebraically closed field.
\end{abstract}

\maketitle



\section*{Introduction}

\noindent
Let $M$ be a commutative monoid with identity in the category of quasi-projective schemes over a field~$k$. Assume that the multiplication map $\mu \colon M \times M \to M$ is proper. On $\CH_*(M)$ we can then define a Pontryagin product, making it into a commutative graded ring.

The first main result of the paper is that there is a canonical PD-structure on the ideal $\CH_{>0}(M)$; see Theorem~\ref{PD-monoid-thm} and Cor.~\ref{PD-ungraded-cor}. The basic geometric idea is that for $Z \subset M$ a closed subvariety of positive dimension, the $d$th divided power of the class $[Z]$ is the class obtained by taking the image of the closed subscheme $\Sym^d(Z) \subset \Sym^d(M)$ under the iterated multiplication map $\Sym^d(M) \to M$. We also have a version of the theorem for graded monoids, and we prove that the PD-structure is functorial.

For an example of a situation where these divided powers naturally appear, consider a smooth proper curve~$C$ with Jacobian~$J$. Examples of monoids to which our result applies are~$J$ (ungraded case) and $\Cbul := \coprod_{n \geq 0}\, C^{[n]}$ with $C^{[n]}$ the $n$th symmetric power of~$C$ (graded case). In \cite{PartII} we have obtained a natural isomorphism $\beta \colon \CH_*(\Cbul) \xrightarrow{\sim} \CH_*(J)\bigl[t\bigr]\bigl\langle u\bigr\rangle$, where $\CH_*(\Cbul)$ and $\CH_*(J)$ are both equipped with the Pontryagin product, and where $R\langle u\rangle$ denotes the PD-polynomial algebra in the variable~$u$ over a commutative ring~$R$. Also we prove that the $\beta$ is a PD-isomorphism with regard to the PD-structures on source and target as defined in the present paper.

A related problem that was raised by B.~Kahn (see \cite{Esnault}) is whether it is possible, for an abelian variety~$X$ over a field $k$, to define a PD-structure on $\CH^{>0}(X)$ with respect to the usual intersection product. As shown in \cite{Esnault} the answer is in general negative. To our knowledge there are, however, no counterexamples with $k$ algebraically closed. Since the Pontryagin product and the usual product on $\CH(X)_{\Q}$
are interchanged by the Fourier transform $\CH(X)_{\Q} \to \CH(X^t)_{\Q}$, 
we are led to ask whether one can define an integral version of the Fourier transform. 
An integral Fourier transform having all the usual properties would allow us to transport the PD-structure on $\CH_*(X)$ for the Pontryagin product to a PD-structure on $\CH^*(X^t)$ for the intersection product. Our second main result is that when $X$ is the Jacobian $J=J(C)$ of a hyperelliptic curve~$C$ of genus~$g$, this plan works up to $2^N$-torsion, with $N = 1 + \lfloor \log_2(3g)\rfloor$. This relies on two
elementary results in intersection theory of $J$ and of powers of $C$ that we establish in 
section~\ref{hyper-sec}. Namely, complementing the results of Collino in~\cite{Collino}
we prove that up to $2$-torsion the intersections of classes of the Brill-Noether loci $W_i$ in~$J$
are given by the same formula as in cohomology. Also, we check that a modified diagonal class on
$C\times C\times C$ introduced by Gross and Schoen in \cite{GS} is $2$-torsion.

The idea for the construction of an integral Fourier transform $\Fourier$ for hyperelliptic Jacobians 
is as follows.
Working with $\Q$-coefficients one can express $\Fourier$ entirely in terms of $*$-products. To be precise, one has $\Fourier(a) = (-1)^g \cdot j_2^*\bigl(\Exp(\tau) * j_{1,*}(a)\bigr)$, where $j_1$, $j_2 \colon J \to J^2$ are the inclusions of the coordinate axes, $\tau \in \CH(J^2)_\Q$ is the class $(-1)^g \cdot \bigl(\Delta_* \Fourier(\theta) - j_{1,*} \Fourier(\theta) - j_{2,*}(\theta)\bigr)$, and where $\Exp(\tau)$ is the $*$-exponential of~$\tau$. (See \eqref{Fourier*Prod} in the text.) Our theorem on divided powers for the Pontryagin product allows us to define the $*$-exponential integrally. The next ingredient we need is the class $\Fourier(\theta)$, which, in general, we do not know how to define integrally. However, if $C$ is hyperelliptic and $\iota \colon C \to J$ is the embedding associated to a Weierstrass point $p_0 \in C$ then we have $\Fourier(\theta) = (-1)^{g+1}\cdot \bigl[\iota(C)\bigr]$. This puts us in a position to define, for hyperelliptic~$C$, an integral endomorphism $\Fourier \colon h(J) \to h(J)$ of the ungraded
motive~$h(J)$. We prove in Theorem~\ref{HEFourier} that it has all the expected properties up to $2^N$-torsion. As a corollary we obtain a natural PD-structure on a large ideal in $\CH(J)$ for the intersection product; see Cor.~\ref{PDIntersProd}. 

Our final result is Theorem~\ref{elliptic-thm}, in which we consider an elliptic curve~$E$ over an algebraically closed field. We prove that it is possible to define a motivic integral Fourier transform $h(E)\to h(E)$ that has all the usual properties up to a factor~$2$ (as made precise in Def.~\ref{mot-def}), and that in general this factor~$2$ cannot be eliminated.


\section{Divided powers on Chow homology}
\label{PD-gen-sec}

Let $X$ be a quasi-projective scheme over a field~$k$. For $d\geq 1$ we define the $d$th symmetric power $\Sym^d(X)$ of~$X$ (over~$k$) as the quotient of $X^d$ by the natural action of the symmetric group~$S_d$. (For the existence of the quotient as a scheme, see \cite{DG}, III, \S~2 or \cite{Mumford}.) For a morphism $a\colon X\to Y$ between quasi-projective $k$-schemes we have an induced morphism $\Sym^d(a) \colon \Sym^d(X)\to \Sym^d(Y)$. If $a$ is a closed embedding then so is $\Sym^d(a)$. (Indeed, it suffices to check this for affine schemes. Then the assertion follows from the remark that a surjection of $k$-vector spaces $V\to W$ induces a surjection on symmetric tensors.)

We recall some elementary, and probably well-known, facts.

\begin{lem}\label{degree-sym-pr-lem}
Let $X$ be a quasi-projective scheme over~$k$.

\textup{(i)} The quotient morphism $q_{d,X} \colon X^d \to \Sym^d(X)$ is finite and $\Sym^d(X)$ is again quasi-projective. 

\textup{(ii)} Assume that $X$ is equidimensional of dimension $n>0$. Then $X^d$ and $\Sym^d(X)$ are equidimensional of dimension~$dn$, and there exists a dense open subset in $\Sym^d(X)$ over which $q_{d,X}$ is \'etale of degree~$d!$.

\textup{(iii)} Assumptions as in~\textup{(ii)}. For nonnegative integers $d_1,\ldots,d_r$ with $d_1+\cdots+d_r=d$, the natural map
$$
\a_{d_1,\ldots,d_r}\colon \Sym^{d_1}(X)\times\cdots\times\Sym^{d_r}(X)\to \Sym^d(X)
$$
is finite, and there is a dense open subset in $\Sym^d(X)$ over which it is \'etale of degree $d!/d_1!\, d_2!\, \cdots d_r!$. For $d$, $e \geq 1$, the natural map $\Sym^d\bigl(\Sym^e(X)\bigr)\to \Sym^{de}(X)$ is finite, and there is a dense open subset in $\Sym^{de}(X)$ over which it is \'etale of degree $(de)!/d!(e!)^d$.
\end{lem}

\begin{proof}
(i) Let $X$ be a quasi-projective scheme over~$k$. Let $q \colon X \to Y$ be the quotient of~$X$ by a finite group~$G$. To see that $q$ is finite one reduces to the situation that $X = \Spec(A)$ and $Y = \Spec(B)$ with $B = A^G$. Then $A$ is integral over~$B$; see \cite{DG}, III, \S~2, 6.1. But also $A$ is a $B$-algebra of finite type, as it is even of finite type over~$k$. This implies that $A$ is finitely generated as a $B$-module. 

To see that the quotient is again quasi-projective, take an ample bundle $L$ on~$X$. Then $\otimes_{g \in G}\; g^*L$ is of the form $q^* M$ for some bundle~$M$ on~$Y$. Because $q$ is finite, the fact that $q^* M$ is ample implies that $M$ is ample. 

(ii) The fact that $X^d$ is equidimensional of dimension~$dn$ is standard, and the corresponding fact for $\Sym^d(X)$ follows from the fact that $q_{d,X}$ is finite and surjective. Next let $U \subset X^d$ be the complement of all big diagonals, i.e., the open subset of points $(x_1,\ldots,x_d)$ with $x_i \neq x_j$ for $i \neq j$. Then $U$ is open and dense in~$X^d$ and $S_d$ acts freely on~$U$. The quotient $V := U/S_d$ is an open dense subscheme of $\Sym^d(X)$ and the quotient morphism is finite \'etale of degree~$d!$ over~$V$.

(iii) Finiteness of the required maps follows from (i). Again let $U \subset X^d$ be the complement of all big diagonals. The quotient morphism $U \to U/S_d$ is finite \'etale of degree~$d!$ and it factors as the composition
$$
U \xrightarrow{\ q_{1}\ } U/(S_{d_1} \times \cdots \times S_{d_r}) \xrightarrow{\alpha_{d_1,\ldots,d_r}} U/S_d
$$
with $q_1$ finite \'etale of degree $d_1!\cdots d_r!$. Moreover, $U/S_d$ is a dense open subscheme of $\Sym^d(X)$ and $U/(S_{d_1} \times \cdots \times S_{d_r})$ is its inverse image in $\Sym^{d_1}(X) \times \cdots \times \Sym^{d_r}(X)$ under $\alpha_{d_1,\ldots,d_r}$. Hence $\alpha_{d_1,\ldots,d_r}$ is finite \'etale of degree $d!/d_1!\cdots d_r!$ over $U/S_d$. For the map $\Sym^d\bigl(\Sym^e(X)\bigr)\to \Sym^{de}(X)$ the argument is similar and we omit the details.
\end{proof}

\begin{rem}\label{flatness-rem}
A general fact that we shall frequently use is the following. Suppose $X$ and $Y$ are irreducible noetherian schemes that are both regular and with $\dim(X) = \dim(Y)$. Then any quasi-finite morphism $X \to Y$ is flat. See for instance \cite{Matsumura}, Thm.~23.1.

As an example of an application, suppose $C$ is a curve over~$k$. Then the symmetric powers $\Sym^m(C)$, for which we shall later use the notation~$C^{[m]}$, are regular. Combining the fact just mentioned with the Lemma, we find that the addition map 
$$
\a_{d_1,\ldots,d_r}\colon \Sym^{d_1}(C)\times\cdots\times\Sym^{d_r}(C)\to \Sym^d(C)
$$
is finite flat of degree $d!/d_1!\, d_2!\, \cdots d_r!$.
\end{rem}

Let $X$ be a quasi-projective scheme over~$k$. If $Z \sub X$ is a closed subscheme that is equidimensional of dimension $m>0$ then $\Sym^d(Z)$ can be viewed as a closed subscheme of $\Sym^d(X)$, and we define 
$$
\gamma_d(Z) := \bigl[\Sym^d(Z)\bigr] \in \CH_{dm}\bigl(\Sym^d(X)\bigr)
$$ 
to be its fundamental class (as in \cite{Fulton}, Section~1.5). For $d=0$ we have $\Sym^0(X)=\Spec(k)$ and we define $\gamma_0(Z) := \bigl[\Sym^0(X)\bigr] = \bigl[\Spec(k)\bigr]$. 

We want to extend the maps $\gamma_d$ to polynomial maps $\CH_m(X) \to \CH_{dm}\bigl(\Sym^d(X)\bigr)$. Modulo torsion this is straightforward, using the formula
\begin{equation}\label{gamma-power-for}
d!\, \gamma_d(Z)=q_{d,X*}[Z^d]
\end{equation}
that follows from (ii) of Lemma~\ref{degree-sym-pr-lem}. Note that the right-hand side of \eqref{gamma-power-for} is just the $d$th Pontryagin power of the cycle $[Z]$, where we define the Pontryagin product 
$$
\CH_*\bigl(\Sym^{d_1}(X)\bigr) \times \CH_*\bigl(\Sym^{d_2}(X)\bigr) \to \CH_*\bigl(\Sym^{d_1+d_2}(X)\bigr)
$$
by setting $x * y := (\a_{d_1,d_2})_* (x \times y)$.

To take torsion classes into account we have to proceed with more care. Let $\zeta = \sum_{j=1}^r\,  n_j [Z_j]$ be a cycle on $X$, where the $n_j$ are integers and the $Z_j$ are mutually distinct closed subvarieties of positive dimensions (possibly unequal). For $d \geq 0$ define $\gamma_d(\zeta) \in \CH_*\bigl(\Sym^d(X)\bigr)$ by
\begin{equation}\label{gammad-def}
\gamma_d(\zeta) := \sum_{d_1 + \cdots + d_r = d}\; n_1^{d_1} \cdots n_r^{d_r} \cdot \gamma_{d_1}(Z_1) * \cdots * \gamma_{d_r}(Z_r)\, .
\end{equation}
Writing $\mathcal{Z}_m(X)$ for the space of cycles of dimension~$m$ and setting
$\mathcal{Z}_{> 0}(X) := \oplus_{m>0}\, \mathcal{Z}_m(X)$, this defines maps
$$
\gamma_d \colon \mathcal{Z}_{> 0}(X) \to \CH_*\bigl(\Sym^d(X)\bigr)\, .
$$

\begin{lem}\label{PD-additivity-lem}
\textup{(i)} If $Z \subset X$ is a closed subvariety with $\dim(Z) >0$ and $d_1,\ldots,d_t$ are nonnegative integers with $d_1+\cdots+d_t = d$ then
$$
\gamma_{d_1}(Z) * \cdots * \gamma_{d_t}(Z) = \frac{d!}{d_1!\, d_2!\, \cdots d_t!} \cdot \gamma_d(Z)\, .
$$

\textup{(ii)} Let $\zeta_1,\ldots,\zeta_t$ be cycles in $\mathcal{Z}_{> 0}(X)$, 
and let $\zeta = \sum_{i=1}^t\,  \zeta_i$. Then 
\begin{equation}\label{gammad-additivity}
\gamma_d(\zeta) = \sum_{d_1 + \cdots + d_t = d}\; \gamma_{d_1}(\zeta_1) * \cdots * \gamma_{d_t}(\zeta_t)\, .
\end{equation}

\textup{(iii)} If $i\colon V\hookrightarrow X$ is a closed immersion and $\zeta$ is a cycle of dimension $m>0$ on~$V$ then
$$
\gamma_d(i_*\zeta) = \Sym^d(i)_*\bigl(\gamma_d(\zeta)\bigr)\, .
$$
(Here $\gamma_d(\zeta) \in \CH_{dm}\bigl(\Sym^d(V)\bigr)$.)

\textup{(iv)} Let $V\sub X$ be a closed subscheme, equidimensional of positive dimension. Then
$$ 
\gamma_d([V]) = \bigl[\Sym^d(V)\bigr]\, ,
$$ 
where we view $\Sym^d(V)$ as a closed subscheme of $\Sym^d(X)$.
\end{lem}

Effectively, (ii) means that \eqref{gammad-def} is still valid if we drop the assumption that the $Z_j$ are mutually distinct.

\begin{proof} (i) This follows from (iii) of Lemma~\ref{degree-sym-pr-lem} applied to~$Z$.

(ii) Let $Z_1,\ldots,Z_r$ be the subvarieties of $X$ that occur as component in some $\zeta_i$. Write $\zeta_i = \sum_{j=1}^r\, n_{ij} [Z_j]$ and let $N_j := \sum_{i=1}^t\, n_{ij}$, so that $\zeta = \sum_{j=1}^r\, N_j [Z_j]$. Then the RHS of \eqref{gammad-additivity} equals
$$
\sum_{\delta}\; n^\delta \cdot \prod_{j=1}^r \prod_{i=1}^t \gamma_{\delta_{ij}}(Z_j)\, ,
$$
where $\delta$ runs over all matrices $(\delta_{ij})$ of size $t \times r$ with $\delta_{ij} \geq 0$ and $\sum \delta_{ij} = d$, and where we write $n^\delta$ for $\prod_{ij}\, n_{ij}^{\delta_{ij}}$. 

Given $\delta$, write $e_j(\delta) := \sum_{i=1}^t\, \delta_{ij}$. Fix nonnegative integers $e_1,\ldots,e_r$ with $e_1+\cdots + e_r =d$ and sum over all matrices~$\delta$ such that $e_j(\delta) = e_j$ for all~$j$. Using~(i) we find that this sum equals
$$
\sum_{\delta \text{\ with\ } \e_j(\delta)=e_j}\; \prod_{j=1}^r\; \frac{n_{1j}^{\delta_{1j}} \cdots n_{tj}^{\delta_{tj}} \cdot e_j!}{\delta_{1j}! \cdots \delta_{tj}!} \cdot \gamma_{e_j}(Z_j) = \prod_{j=1}^r N_j^{e_j}\cdot  \gamma_{e_j}(Z_j)\, .
$$
Summing over all $(e_1,\ldots,e_r)$ we get the LHS of \eqref{gammad-additivity}.

(iii) This follows immediately from the definitions, using that the push-forward via the map $\Sym^d(i) \colon \Sym^d(V) \to \Sym^d(X)$ respects Pontryagin products.

(iv) By (iii) it is enough to check this in the case $V=X$. But in this case we have to check the equality in $\CH_{dm}\bigl(\Sym^d(V)\bigr)$, where $m=\dim V$. Since $\Sym^d(V)$ is equidimensional of dimension~$dm$, this group has no torsion, so the equality follows from \eqref{gamma-power-for} and (ii) of Lemma~\ref{degree-sym-pr-lem}.
\end{proof}
 
\begin{prop}
The maps $\gamma_d \colon \mathcal{Z}_{> 0}(X) \to \CH_*\bigl(\Sym^d(X)\bigr)$ factor modulo rational equivalence. 
\end{prop}

\begin{proof}
Let $V\sub X\times\P^1_k$ be a closed subvariety of dimension $>1$, such that the projection $f\colon V\to\P^1$ is dominant. Then $f$ is surjective and flat of relative dimension $>0$. If $y \in \P^1(k)$ is a $k$-rational point of~$\P^1$, write $V_y$ for the (scheme-theoretic) fibre of~$f$ over~$y$, viewed as a closed subscheme of~$X$. Note that $V_y$ is non-empty with $\dim(V_y) = \dim(f) > 0$. 
We claim that in this situation we have $\gamma_d[V_0] = \gamma_d[V_\infty]$. 
To prove this, consider the relative symmetric power $\Sym^d(V/\P^1)$ (see \cite{SGA4-17}, 5.5). We can view it as a closed subscheme in $\Sym^d(X\times\P^1/\P^1) = \Sym^d(X)\times\P^1$. The flatness of $f$ implies that $\Sym^d(V_y)$ is the fibre of $\Sym^d(V/\P^1)$ over the point~$y$; see \cite{SGA4-17}, 5.5.2.5 and 5.5.2.7, or \cite{BLR}, Section~9.3. Hence, by (iv) of Lemma~\ref{PD-additivity-lem}, 
$$
\gamma_d[V_0] - \gamma_d[V_\infty] = \phi_*\bigl([\Sym^d(V/\P^1)_0] - [\Sym^d(V/\P^1)_\infty]\bigr)\, ,
$$
where $\phi\colon \Sym^d(V/\P^1)\to\Sym^d(X)$ is the projection.
Since $\Sym^d(V/\P^1)$ is flat over $\P^1$ (see \cite{SGA4-17}, 5.5.2.4), the latter class is rationally equivalent to zero by \cite{Fulton}, Theorems 1.7 and~1.4. 

By Lemma~\ref{PD-additivity-lem}(ii), $H := \cap_{d \geq 1}\, \Ker(\gamma_d)$ is a linear subspace of $\mathcal{Z}_{>0}(X)$, and for $\zeta$, $\zeta^\prime \in \mathcal{Z}_{>0}(X)$ we have $\gamma_d(\zeta) = \gamma_d(\zeta^\prime)$ for all~$d$ if and only if $\zeta - \zeta^\prime \in H$. As just shown, for all subvarieties $V\sub X\times\P^1$ of dimension $>1$, dominant over $\P^1$, and all $d\geq 0$ we have $\gamma_d[V_0]=\gamma_d[V_\infty]$; hence, $[V_0]-[V_\infty]\in H$. It is known that the subspace of $\mathcal{Z}_{>0}(X)$ of classes that are rationally equivalent to zero is spanned by classes of this form. (See \cite{Fulton}, Prop.~1.6.) Hence this subspace is contained in~$H$ and we are done.
\end{proof}

Thus, for every quasi-projective $k$-scheme $X$ we have well-defined maps
$$
\gamma_d \colon \CH_{>0}(X)\to \CH_*\bigl(\Sym^d(X)\bigr)\, .
$$
Let us check the compatibility of these maps with push-forwards.

\begin{prop}\label{gamma-fun-prop} 
Let $f\colon X\to Y$ be a proper morphism of quasi-projective $k$-schemes.
Then for all $x\in\CH_{>0}(X)$ and all $d\geq 0$ one has
$$
\Sym^d(f)_*\bigl(\gamma_d(x)\bigr) = \gamma_d(f_*x)\, .
$$
\end{prop}

\begin{proof}
By (iii) and (iv) of Lemma \ref{PD-additivity-lem} it suffices to show that
\begin{equation}\label{gamma-fun-formula}
\Sym^d(f)_*\bigl[\Sym^d(X)\bigr]=\gamma_d\bigl(f_*[X]\bigr)\, ,
\end{equation} 
where we may further assume that $X$ is irreducible and that $Y = f(X)$ is the scheme-theoretic image of~$f$. If $\dim(Y) < \dim(X)$ then both sides of \eqref{gamma-fun-formula} are zero. Hence, we only have to consider the case when $X$ and $Y$ are irreducible of the same dimension $m>0$. In this case \eqref{gamma-fun-formula} readily follows from \eqref{gamma-power-for} because $\CH_{md}(\Sym^d(Y))$ has no torsion.
\end{proof}

Let $k$ be a field, and let $(M_n)_{n\ge 0}$ be a commutative graded monoid in the category of quasi-projective $k$-schemes. This means that each $M_n$ is a quasi-projective $k$-scheme, and that we have product maps $\mu_{m,n} \colon M_m\times M_n \to M_{m+n}$ satisfying commutativity and associativity. We further assume that there is a point $e\in M_0(k)$ which is a unit for these products and that all product maps $\mu_{m,n}$ are proper. Then we can define the Pontryagin product on 
$$
\CH_*(M_{\bullet}) := \bigoplus_{n\ge 0}\, \CH_*(M_n)
$$
by the standard formula: $x * y = (\mu_{m,n})_*\bigl(x \times y\bigr)$ for $x \in \CH_*(M_m)$ and $y \in \CH_*(M_n)$. Using the operations $\gamma_d$ we can define a PD-structure on the ideal $\CH_{>0}(M_{\bullet})$. Before we state and prove the main result, note that the iterated multiplication morphism $\mu_{n,n,\ldots,n} \colon M_n^d\to M_{dn}$ factors through a proper map $p_d \colon \Sym^d(M_n)\to M_{dn}$. For $d=0$ we set $p_0=e \colon\Sym^0(M_n) = \Spec(k) \to M_0$.

\begin{thm}\label{PD-monoid-thm} 
For a commutative graded monoid $(M_n)$ with identity and with proper product morphisms, the maps
$$
\gamma^M_d\colon \CH_{>0}(M_n) \to \CH_*(M_{dn})
$$
given by $x\mapsto p_{d,*}\gamma_d(x)$ extend uniquely to a PD-structure $\{\gamma_d^M\}_{d \geq 0}$ on the ideal $\CH_{>0}(M_{\bullet})\sub \CH_*(M_{\bullet})$. If $(f_n \colon M_n\to M'_n)$ is a homomorphism of two such monoids such that all morphisms $f_n$ are proper then the push-forward map $(f_\bullet)_* \colon \CH_*(M_{\bullet}) \to \CH_*(M'_{\bullet})$ is a PD-homomorphism.
\end{thm}

\begin{proof}
Given $x = \sum_{n\geq 0} x_n$ with $x_n \in \CH_{> 0}(M_n)$, zero for almost all~$n$, define
$$
\gamma_d^M(x) := \sum_{d_0 + d_1 + \cdots = d}\; \gamma_{d_0}^M(x_0) * \gamma_{d_1}^M(x_1) * \cdots 
$$
It is clear that $\gamma_d(\lambda x) = \lambda^d \cdot \gamma_d(x)$ and it easily follows from Lemma~\ref{PD-additivity-lem} that we have 
\begin{equation}\label{gammad(x+y)}
\gamma_d^M(x+y) = \sum_{i+j=d}\, \gamma_i^M(x) *\gamma_j^M(y)
\end{equation}
for all $x$, $y \in \CH_{>0}(M_{\bullet})$ and all $d \geq 0$. It remains to be shown that for all $x \in \CH_{> 0}(M_{\bullet})$ and all integers $d$, $e \geq 0$ we have the following relations:
\begin{equation}\label{PD-check-1}
\gamma_d^M(x) * \gamma_e^M(x) = \binom{d+e}{d}\cdot \gamma_{d+e}^M(x)
\end{equation}
and
\begin{equation}\label{PD-check-2}
\gamma_d^M\bigl(\gamma_e^M(x)\bigr) = \frac{(de)!}{d! (e!)^d}\cdot \gamma_{de}^M(x)\, .
\end{equation}
Moreover, because \eqref{gammad(x+y)} holds, it suffices to verify this for classes $x = [Z]$ with $Z \subset M_n$ a closed subvariety, $\dim(Z) >0$. In this case \eqref{PD-check-1} follows from Lemma~\ref{PD-additivity-lem}(i) and \eqref{PD-check-2} follows from the second assertion of 
Lemma \ref{degree-sym-pr-lem}(iii).

The second assertion of the theorem follows from Proposition \ref{gamma-fun-prop}.
\end{proof}

For example, if $X$ is a quasi-projective $k$-scheme then the above construction gives a PD-structure on the ideal $\CH_{>0}\bigl(\Sym^{\bullet}(X)\bigr) \sub \CH_*\bigl(\Sym^{\bullet}(X)\bigr)= \oplus_{d \geq 0}\; \CH_*\bigl(\Sym^d(X)\bigr)$.

The case where $M_n = \emptyset$ for all $n>0$ gives the following ungraded version of the theorem.

\begin{cor}\label{PD-ungraded-cor}
Let $M$ be a commutative monoid with identity in the category of quasi-projective $k$-schemes, such that the product morphism $\mu\colon M\times M\to M$ is proper. Let $p_d\colon \Sym^d(M)\to M$ be the morphism induced by the iterated multiplication map $M^d \to M$. Then the maps $\gamma_d^M \colon \CH_{>0}(M)\to \CH_*(M)$ defined by $x \mapsto p_{d,*}\gamma_d(x)$ define a PD-structure on the ideal $\CH_{>0}(M) \sub \CH_*(M)$.
\end{cor}

For example, if $A$ is an abelian variety over $k$ then this construction gives a PD-structure on $\CH_{>0}(A)\sub\CH_*(A)$.

\begin{cor} 
Assume that the field $k$ is algebraically closed. Then for an abelian variety $A$ over $k$ there is a canonical PD-structure (with respect to the Pontryagin product)
on the augmentation ideal in $\CH_*(A)$, i.e., the ideal generated by $\CH_{>0}(A)$ together with the $0$-cycles of degree~$0$.
\end{cor}

\begin{proof} 
Let $I\sub \CH_0(A)$ denote the ideal of $0$-cycles of degree~$0$ on~$A$. As clearly $I \cap \CH_{>0}(A) = (0)$, it is enough to prove that $I$ has a unique PD-structure. But this follows from the well-known fact that $I^{*2}$ is a $\Q$-vector space; see \cite{Beauv}, \S~4, where we use Milne's refinement in \cite{Milne} of the results of Rojtman in~\cite{Roitman}.
\end{proof}

\section{Intersection theory on hyperelliptic Jacobians}
\label{hyper-sec}

\subsection{}
\label{HECurve-setup}
Let $C$ be a hyperelliptic curve of genus $g \geq 1$ over a field~$k$, and let $p_0 \in C$ be a Weierstrass point defined over $k$. Let $J := \Pic^0_{C/k}$ be the Jacobian, and consider the embedding $\iota \colon C \to J$ given on points by $x\mapsto \bigl[\OO_C(x-p_0)\bigr]$.

We denote the hyperelliptic involution by $x \mapsto \bar{x}$. If $D$ is a divisor on~$C$ then by $\bar{D}$ we denote the image of~$D$ under the hyperelliptic involution. The double covering $\nu \colon C \to \P^1$ induces for each $n \geq 0$ a morphism $\nu^* \colon \P^n \to C^{[2n]}$.

If $A$ is a divisor of degree $i \geq 0$ on~$C$, let $s(A) \colon C^{[i]} \to J$ be the morphism given by $D \mapsto  \bigl[\OO_C(D-A)\bigr]$. Also define $w(A) := s(A)_*\bigl[C^{[i]}\bigr] \in \CH_i(J)$. Clearly, 
$s(A)$, and hence also~$w(A)$, only depend on the class of~$A$ in~$\CH(C)$. 
In particular, $w(i\cdot p_0)$ is the usual class~$[W_i]$ of the locus 
$$W_i = \bigl\{L \in J \bigm| h^0(L(ip_0))>0  \bigr\}.$$ 
Note further that $w(A)$ is algebraically equivalent to~$[W_i]$ if $i=\deg(A)$.

It was proved by Collino in \cite{Collino} that $n!\, [W_{g-n}]=\th^n$, where $\th=[W_{g-1}]$. This means that with rational coefficients the classes $[W_{g-i}]$ can be recovered
from~$\th$. However, one may still hope to get more information about these classes working with
integer coefficients. For example, Collino proves in \cite{Collino} a more precise formula in $\CH(J)$, namely that $[W_{g-n}]\cdot\th=(n+1)\cdot [W_{g-n-1}]$ for $n\le g-1$. We are going to show that, up to some $2$-torsion classes, the intersections of the loci $[W_i]$ in $\CH(J)$ are what we expect them to be; see Theorem~\ref{2WW=2W}.

\begin{lem}\label{W-div-lem}
Suppose $Q$ and $Q^\prime$ are divisors on~$C$ of the same degree that are both supported on $k$-rational Weierstrass points. Then
$$
2 \cdot w(Q) = 2 \cdot w(Q^\prime)\, .
$$
In particular, if $\deg(Q) = i \geq 0$ then $2\cdot w(Q) = 2\cdot [W_i]$.
\end{lem}

\begin{proof}
Let $a \in J[2]$ be the class of $Q - Q^\prime$. We have $s(Q^\prime) = t_a \circ s(Q)$, where $t_a \colon J \to J$ is the translation over~$a$. Hence $w(Q^\prime) = [a] * w(Q)$, with $[a] \in \CH_0(J)$ the class of~$a$. 

Writing $e \in J(k)$ for the origin, we claim that $2[a] * y = 2[e] * y = 2y$ for all $y \in \CH(J)$. To see this, write $Q-Q^\prime = \sum_{i=1}^m\, (q_i-r_i)$ where the $q_i$ and~$r_i$ are Weierstrass points. If $m=0$ we have $[a] = [e]$ and there is nothing to prove. If $m \geq 1$, let $a_1 \in J[2]$ be the class of $q_1 - r_1$ and let $a^\prime$ be the class of $\sum_{i=2}^m\, (q_i-r_i)$. Writing $\dot{+}$ for the group law in~$J$ (in order to avoid confusion), we have $a = a_1 \dot{+} a^\prime$; hence $[a] * y = [a_1] * [a^\prime] * y$. By induction on~$m$ we may assume that $2[a^\prime] * y = 2y$ for all~$y$; hence we find that $2[a] * y = 2[a_1] * y$. But the element $[a_1] - [e]$ in $\CH_0(J)$ is $2$-torsion, because it is the push-forward of the $2$-torsion class $[q_i]-[r_i] \in \CH_0(C)$ under the morphism $C \to J$ given by $x \mapsto \bigl[\OO_C(x-r_1)\bigr]$. So indeed $2[a] * y = 2y$ for all~$y$, and this proves that $2w(Q^\prime) = 2w(Q)$.
\end{proof}

\begin{thm}
\label{2WW=2W}
In the above situation assume also that there exist $g$ distinct Weierstrass points on $C$ defined over $k$. If $m \geq 0$ and $n \geq 0$ are integers with $m + n \leq g$ then we have the relation
$$
2 \cdot [W_{g-m}] \cdot [W_{g-n}] = 2 \cdot \binom{m+n}{m} \cdot [W_{g-m-n}]
$$
in $\CH(J)$.
We also have
$$
[W_{g-m}] \cdot [W_{g-n}] =  \binom{m+n}{m} \cdot [W_{g-m-n}]
$$
in $\CH(J)/{\sim}_\alg$, the group of cycles modulo algebraic equivalence.
\end{thm}

\begin{proof}
By assumption there exist $g$ mutually distinct Weierstrass points on~$C$ defined over $k$; call them
$$
q_1,\ldots,q_{g-m-n},r_1,\ldots,r_n,r_1^\prime,\ldots,r_m^\prime\, ,
$$
and write 
$$
Q := \sum q_i\, ,\quad R := \sum r_i\, ,\quad \text{and}\quad R^\prime := \sum r_i^\prime\, .
$$
Define $Y$ by the fibre product diagram
$$
\begin{matrix}
Y & \xrightarrow{\quad\qquad} & C^{[g-m]}\\
\Big\downarrow && \Big\downarrow\rlap{$\scriptstyle s(Q+R)$}\\
C^{[g-n]} &\xrightarrow{\ s(Q+R^\prime)\ } & J\rlap{\; .}
\end{matrix}
$$
The points of~$Y$ are the pairs $(D,D^\prime) \in C^{[g-m]} \times C^{[g-n]}$ such that $D+R^\prime \sim D^\prime + R$.

Suppose we have integers $d \geq 0$ and $e \geq 0$ with $d+e \leq g-m-n$. Write $R$ and~$R^\prime$ as sums of effective divisors, say
\begin{equation}
\label{RRprimeDec}
R = S+T \qquad\text{and}\qquad R^\prime = S^\prime + T^\prime
\end{equation}
such that $e + \deg(S) = d + \deg(S^\prime)$. To these data we associate a morphism
$$
f = f_{d,e,S,S^\prime} \colon \P^d \times \P^e \times C^{[g-m-n-d-e]} \to Y
$$
given on points by
$$
(A,B,D) \mapsto \bigl(\nu^*(A)+T+S^\prime+D,\nu^*(B)+S+T^\prime+D\bigr)\, .
$$
Note that the right hand term is indeed a point of $C^{[g-m]} \times C^{[g-n]}$. Note further that
\begin{align*}
\bigl(\nu^*(A)+T+S^\prime+D\bigr) + R^\prime &\sim \bigl(d+\deg(S^\prime)\bigr)\cdot g^1_2 + T + T^\prime + D \\
&= \bigl(e+\deg(S)\bigr)\cdot g^1_2 + T + T^\prime + D \sim \bigl(\nu^*(B)+S+T^\prime+D\bigr) + R\, ,
\end{align*}
so $f$ does indeed define a morphism to~$Y$.

Next we prove that the disjoint union of all~$f$ is surjective on geometric points. To see this, suppose we have $(D,D^\prime)$ with $D+R^\prime \sim D^\prime+R$. Let $D_1$ be the largest effective divisors of the form $D_1 = \nu^*(A)$ such that $D_1 \leq D$. Write $D-D_1 = D_2 + D_3 + D_4$, where $D_2$ has support in $\{r_1,\ldots,r_n\}$, where $D_3$ has support in $\{r_1^\prime,\ldots,r_m^\prime\}$, and $D_4$ has support outside these points. By our choice of~$D_1$, all points in $D_2$ and~$D_3$ have multiplicity~$1$; hence there are unique decompositions $R= S_1+T_1$ and $R^\prime = S_1^\prime +T_1^\prime$ such that $D_2 = T_1$ and $D_3 = S_1^\prime$. In total this gives us
$$
D = \nu^*(A) + T_1 + S_1^\prime + D_4\, .
$$
In a similar way, we obtain decompositions $R= S_2+T_2$ and $R^\prime = S_2^\prime +T_2^\prime$ and an expression
$$
D^\prime = \nu^*(B) + S_2 + T_2^\prime + D_4^\prime\, .
$$
The relation $D+R^\prime \sim D^\prime +R$ can be rewritten as
$$
\bigl(\deg(A)+\deg(S_1^\prime)\bigr)\cdot g^1_2 + T_1 + T_1^\prime + D_4 \sim \bigl(\deg(B) + \deg(S_2)\bigr)\cdot g^1_2 + T_2 + T_2^\prime + D_4^\prime\, .
$$
Now we use that if $E$ is an effective divisor of degree $\leq g$ that does not contain the $g^1_2$ on~$C$ then $h^0(E)=1$. Applying this with $E= T_1 + T_1^\prime + D_4$ and $E= T_2 + T_2^\prime + D_4^\prime$, it follows that $\deg(A)+\deg(S_1^\prime)= \deg(B) + \deg(S_2)$ and that $T_1 + T_1^\prime + D_4 = T_2 + T_2^\prime + D_4^\prime$. This last identity implies that $T_1=T_2$ and $T_1^\prime = T_2^\prime$ (hence also $S_1=S_2$ and $S_1^\prime=S_2^\prime$), and that $D_4=D_4^\prime$. This proves that $(D,D^\prime)$ is in the image of some $f_{d,e,S,S^\prime}$.

The morphisms $f_{d,e,S,S^\prime}$ are finite and generically injective. Their images $Y(d,e,S,S^\prime) \subset Y$ are mutually distinct and all have the expected dimension $g-m-n$. The components $Y(d,e,S,S^\prime)$ with $d>0$ or $e>0$ push forward to~$0$ in $\CH(J)$. The components $Y(0,0,S,S^\prime)$ push forward to $w(Q+S-S^\prime)$. Note that there are precisely $\binom{m+n}{m}$ such components $Y(0,0,S,S^\prime)$, as a choice of decompositions~\eqref{RRprimeDec} with $\deg(S)=\deg(S^\prime)$ corresponds to the choice of a subset $J \subset \{r_1,\ldots,r_n,r_1^\prime,\ldots,r_m^\prime\}$ with $\# J = m$. (Take $S$ to be the sum of the points $r_i$ that are in~$J$ and $S^\prime$ the sum of the points $r_j^\prime$ that are not in~$J$.) It follows that 
$$
w(Q+R) \cdot w(Q+R^\prime) = \sum_{(S,S^\prime)}\, \nu(S,S^\prime) \cdot w(Q+S-S^\prime)\, ,
$$
where $\nu(S,S^\prime)$ is the multiplicity of the component $Y(0,0,S,S^\prime)$ in~$Y$. 

In cohomology we know that $[W_{g-m}] \cdot [W_{g-n}] = \binom{m+n}{m} \cdot [W_{g-m-n}]$. As there are $\binom{m+n}{m}$ choices for $(S,S^\prime)$ and $\nu(S,S^\prime) \geq 1$, it follows that $\nu(S,S^\prime) = 1$ for all $(S,S^\prime)$. The statement now follows from Lemma~\ref{W-div-lem}.
\end{proof}

\subsection{Cubical relation}
\label{ModDiagCycDef}
We keep the notation of \ref{HECurve-setup}. Consider the following $1$-cycle on~$C^3$:
\begin{multline*}
\De_e := i_{(x,x,x),*}[C] - i_{(x,x,p_0),*}[C] - i_{(x,p_0,x),*}[C] - i_{(p_0,x,x),*}[C]\\ + i_{(x,p_0,p_0),*}[C] 
+ i_{(p_0,x,p_0),*}[C] +  i_{(p_0,p_0,x),*}[C]\, ,
\end{multline*}
where $i_{(x,x,p_0)} \colon C \to C^3$ is the map given by $x \mapsto (x,x,p_0)$, and likewise for the other maps. This cycle was studied by Gross and Schoen in~\cite{GS}, where they proved that $6\De_e=0$; see \cite{GS}, Cor.~4.9. 

\begin{prop}
\label{cubical-prop} 
We have $2\cdot \De_e=0$ in $\CH_1(C^3)$.
\end{prop}

\begin{proof} Let $\Delta^- \subset C^2$ be the graph of the hyperelliptic involution, and let $\pi \colon C \to \P^1$ be the double covering. On $\P^1 \times \P^1$ we have the relation $[\Delta] \sim \bigl[\{\pt\} \times \P^1\bigr] + \bigl[\P^1 \times \{\pt\}\bigr]$. Pulling back via $\pi \times \pi$ we get
\begin{equation}
\label{DeltaDelta-}
[\Delta] + [\Delta^-] = 2 \cdot \bigl[\{p_0\} \times C\bigr] + 2\cdot \bigl[C \times \{p_0\}\bigr]
\end{equation}
in $\CH(C^2)$. Now
write $Z(x,x,\bar{x}) \in \CH_1(C^3)$ for the class of the image of the morphism $C \to C^3$ given on points by $x \mapsto (x,x,\bar{x})$. In a similar way we define classes such as 
$Z(x,x,p_0)$ or $Z(p_0,x,\bar{x})$. Now consider the following morphisms $C^2 \to C^3$, and take the push-forwards of~\eqref{DeltaDelta-}:
\begin{alignat*}{2}
&\text{morphism $C^2 \to C^3$:}\qquad\quad &&\text{resulting relation:}\\
\noalign{\smallskip}
&(x,y) \mapsto (x,x,y)      
&& Z(x,x,x) + Z(x,x,\bar{x}) = 2 \cdot Z(x,x,p_0) + 2\cdot Z(p_0,p_0,x)\\ 
&(x,y) \mapsto (x,y,x)       
&& Z(x,x,x) + Z(x,\bar{x},x) = 2\cdot Z(x,p_0,x) + 2\cdot Z(p_0,x,p_0)\\
&(x,y) \mapsto (x,y,\bar{y}) 
&& Z(x,x,\bar{x}) + Z(x,\bar{x},x) = 2\cdot Z(x,p_0,p_0) + 2\cdot Z(p_0,x,\bar{x})
\end{alignat*}
Taking the sum of the first two relations minus the third, we get
\begin{multline*}
2 \cdot \bigl(Z(x,x,x) - Z(x,x,p_0) - Z(x,p_0,x) - Z(p_0,p_0,x) - Z(p_0,x,p_0)\\ + Z(x,p_0,p_0) + Z(p_0,x,\bar{x})\bigr) = 0\, .
\end{multline*}
Again using \eqref{DeltaDelta-}, now on the last two coordinates, gives
\begin{multline}
2 \cdot \bigl(Z(x,x,x) - Z(x,x,p_0) - Z(x,p_0,x) - Z(p_0,x,x)\\ + Z(x,p_0,p_0) + Z(p_0,x,p_0) + Z(p_0,p_0,x) \big) = 0\, ,
\end{multline}
\end{proof}

\begin{rem}\label{cubical-rem}
In general the factor~$2$ in Proposition \ref{cubical-prop} is needed,
i.e., we cannot expect the class $\De_e$ to be trivial in $\CH_1(C^3)$. 
To see this we use an argument that was inspired by \cite{BST}, Section~4. We work over a field~$k$ of characteristic $\neq 2$ in which $-1$ is a square. Consider a hyperelliptic curve~$C$ given by an affine equation $y^2 = x \cdot P(x^2)$. We have an automorphism~$\alpha$ of order~$4$ given by $(x,y) \mapsto \bigl(-x,\sqrt{-1}y\bigr)$ that has two fixed points, $O = (0,0)$ and~$\infty$. Consider the morphism $f\colon C^2 \to C^3$ given by $(a,b) \mapsto \bigl(a,\alpha(a),b\bigr)$, and look at the map $\CH_1(C^3) \to \CH_0(C)$ given by $\xi \mapsto \pr_{2,*}f^*(\xi)$. The image under this map of the class~$\De_e$ equals $[O]-[\infty]$, which has order~$2$. Note that the assumption that $k$ contains a square root of~$-1$ is in fact superfluous, for if $\De_e$ is nonzero over an extension of~$k$ then clearly also $\De_e \neq 0$ in $\CH(C^3)$. 
\end{rem}


\section{Integral aspects of Fourier duality}
\label{IntegralFourier}

Let $X$ be an abelian variety over a field. It has been asked (see \cite{Esnault}) whether there is a PD-structure on $\CH^*(X)$ for the usual intersection product. As shown by Esnault in loc.\ cit.\ Remark~4.1, in general the answer is~\emph{no}. 

Comparison with our results naturally leads to the question whether or not one can define an integral version of the Fourier transform. As the following example shows, in general the answer is again negative.

\subsection{Example} 
We give an example of a Jacobian~$J$ over a field~$k$ such that the Fourier transform $\Fourier \colon \CH(J)_\Q \to \CH(J)_\Q$ (identifying $J^t$ with~$J$ via the canonical principal polarization $\lambda \colon J \xrightarrow{\sim} J^t$) does not map $\CH(J)/\text{torsion}$ into itself. For this, let $k$ and $(C,p_0)$ be respectively the function field and the generic point of the moduli stack $\mathcal{M}_{g,1}$ (which generically is a scheme). It is known that the Jacobian $J$ has no $k$-rational torsion. (This follows from the fact that the mapping class group $\Gamma^1_g$ surjects to $\Sp_{2g}(\Z)$. See \cite{Ekedahl}, Thm.~2.1 for a purely algebraic argument.) Hence also $\Pic(J)$ has no torsion. On the other hand, the component of $-\Fourier\bigl([\iota(C)]\bigr)$ in codimension~$1$ is the class $\theta \in \CH^1(J)_\Q$ of a symmetric theta divisor. By our choice of~$C$ this class does not lie in $\Pic(J)$, and since $\Pic(J)$ has no torsion this implies that $\Fourier\bigl([\iota(C)]\bigr)$ does not lie in $\CH(J)/\text{torsion} \subset \CH(J)_\Q$.

\subsection{{}}
It is not known to us if the obstruction to the existence of an integral Fourier transform is of an arithmetic nature. It is possible that over an algebraically closed field $\Fourier$ does preserve $\CH(J)/\text{torsion}$. (See however Thm.~\ref{elliptic-thm} below.)

Our goal in this section is to prove that for hyperelliptic Jacobians, one has an integral Fourier duality ``up to a factor $2^N$,'' with $N = {1+\lfloor\log_2(3g)\rfloor}$. See Thm.~\ref{HEFourier} below for the precise statement. One may note that \cite{Beauv}, Prop.~$3^\prime$ already gives a kind of integral Fourier transform ``up to a finite factor~$M$,'' but in that result there is no further information about the factor~$M$ that is needed.

Our result allows us to define a PD-structure (for the intersection product) on $2^N \cdot \ag \subset \CH(J)$, where $\ag := \Ker\bigl(\CH(J) \to \CH^0(J) = \Z\bigr)$ is the augmentation ideal. See Cor.~\ref{PDIntersProd}.

\subsection{Motivic integral Fourier transforms}
Let $\Mot(k)$ be the category of effective motives over a field~$k$ with respect to ungraded correspondences. If $X$ is a smooth projective variety over~$k$, let $h(X)$ be the associated motive. The morphisms from $h(X)$ to~$h(Y)$ are the elements of $\CH(X\times Y)$. There is a tensor structure on $\Mot(k)$ with $h(X) \otimes h(Y) = h(X \times Y)$. 

A morphism $f \colon X \to Y$ of smooth projective varieties over~$k$ gives rise to morphisms of motives $[\Gamma_f]$ from $h(X)$ to~$h(Y)$ and $\bigl[{}^t\Gamma_f\bigr]$ from $h(Y)$ to~$h(X)$; here $\Gamma_f$ is the graph of~$f$ and ${}^t \Gamma_f$ is its transpose. If there is no risk of confusion we shall usually simply write $f_*$ for~$[\Gamma_f]$ and $f^*$ for~$\bigl[{}^t\Gamma_f\bigr]$.

For an (ungraded) correspondence $c\in \CH(X\times Y)$ we have an induced homomorphism of Chow groups
$$
c_*\colon \CH(X)\to\CH(Y) \qquad\text{given by}\quad  a\mapsto \pr_{Y,*}\bigl(\pr_X^*(a)\cdot c\bigr)\, .
$$
Note that $[\Gamma_f]_* = f_*$ and $\bigl[{}^t\Gamma_f\bigr]_* = f^*$.

Let $X$ be an abelian variety, and let $a\in \CH(X)$. Then $\Delta_*(a)$ and $(\pr_2 - \pr_1)^*\bigl(a\bigr)$ are classes in $\CH(X^2)$ and so define endomorphisms of~$h(X)$.
The induced endomorphisms of $\CH(X)$ are the maps $b \mapsto a \cdot b$ and $b \mapsto a * b$, respectively. If there is no risk of confusion, we simply write $(a \cdot -)$ for $\Delta_*(a)$ and $(a * -)$ for $(\pr_2-\pr_1)^*\bigl(a\bigr)$, viewed as endomorphisms of~$h(X)$.

The usual Fourier transform for cycles on abelian varieties is induced by a morphism in the category $\Mot_{\Q}(k)$ of ungraded motives with \emph{rational} coefficients. For a smooth projective~$X$ we denote by $h(X)_{\Q}$ the corresponding motive with rational coefficients. If $X$ is an abelian variety then $\Fourier_X \colon \CH(X)_{\Q} \to \CH(X^t)_{\Q}$ is the homomorphism $\bigl[\chern(\PP_X)\bigr]_*$, where $\PP_X$ is the Poincar\'e bundle on $X \times X^t$.

\begin{defi}\label{mot-def} 
Let $X$ be an abelian variety, and let $\phi \colon X \xrightarrow{\sim} X^t$ be an isomorphism from~$X$ to its dual. (We do not require $\phi$ to be a polarization.) Let $d$ be a positive integer. We say that a morphism $\Fourier \colon h(X)\to h(X)$ in $\Mot(k)$ is a \emph{motivic integral Fourier transform of $(X,\phi)$ up to factor~$d$} if the following properties hold.

\begin{itemize}
\item[(i)] The induced morphism $h(X)_{\Q}\to h(X)_{\Q}$ is the composition of the
usual Fourier transform with the isomorphism $\phi^* \colon h(X^t)_{\Q} \xrightarrow{\sim} h(X)_{\Q}$.
\item[(ii)] We have $d\cdot\Fourier^2 = d\cdot (-1)^g\cdot [-1]_*$ as morphisms from $h(X)$ to~$h(X)$.
\item[(iii)] We have $d\cdot \Fourier \circ m_* = d\cdot \De^* \circ (\Fourier \ot\Fourier)$ as morphisms from $h(X)\ot h(X)$ to~$h(X)$; here $m \colon X \times X \to X$ is the addition on~$X$.
\end{itemize}
\end{defi}

If (ii) holds then (iii) is equivalent with the condition
\begin{equation}\label{switch-F-id}
d\cdot \Fourier \circ\De^*=d\cdot (-1)^g\cdot m_*\circ(\Fourier \ot\Fourier)\, .
\end{equation}
Together with (iii) this is a motivic version of the fact that $\Fourier$ interchanges the intersection product and the Pontryagin product (up to factor~$d$).

\begin{defi}
\label{E(a)-Geps-def}
(i) If $X$ is an abelian variety and $a \in \CH_{>0}(X)$, write
$$
\Exp(a) := \sum_{n \geq 0}\, a^{[n]}
$$
for the $*$-exponential of~$a$. Here of course the elements~$a^{[n]}$ are the divided powers of~$a$ as defined in Section~\ref{PD-gen-sec}.

(ii) Let $X$ and $Y$ be abelian varieties. For a class $\epsilon \in \CH(X\times Y)$, define 
a morphism $\GG_\epsilon \colon h(X) \to h(Y)$ by 
$$
\GG_\epsilon(x) := j_2^* \circ (\epsilon * -) \circ j_{1,*}\, ,
$$
where $j_1 \colon X \to X \times Y$ and $j_2 \colon Y \to X \times Y$ are the morphisms given by $j_1(x) = (x,0)$ and $j_2(y) = (0,y)$.
\end{defi}

\begin{lem}
\label{Pontr-ker-lem}
Let $X$ and $Y$ be abelian varieties.

\textup{(i)} If $f \colon X \to Y$ is a homomorphism and $a \in \CH(X)$ then $f^* \circ \bigl(f_*(a) * -\bigr) = (a * -) \circ f^*$ as morphisms from $h(Y)$ to~$h(X)$.

\textup{(ii)} The map 
$$
\CH(X\times Y) \to \Hom_{\Mot(k)}\bigl(h(X),h(Y)\bigr) \qquad\text{given by}\qquad \eps \mapsto \GG_\eps
$$
is an isomorphism. If $Z$ is a third abelian variety then for
$\epsilon' \in \CH(Y\times Z)$ we have
\begin{equation}\label{Pont-ker-comp}
\GG_{\epsilon^\prime} \circ \GG_\epsilon = \GG_{\epsilon^\pprime},
\end{equation}
where $\epsilon^\pprime \in \CH(X\times Z)$ is given by $\epsilon^\pprime = j_{13}^*\bigl(j_{12,*}(\epsilon) * j_{23,*}(\epsilon^\prime)\bigr)$. Here $j_{12} \colon X \times Y \to X \times Y \times Z$ is given by $(x,y) \mapsto (x,y,0)$ and $j_{13}$ and~$j_{23}$ are defined similarly. 
\end{lem}

On Chow groups (i) gives
\begin{equation}
\label{ProjFormDual}
f^*\circ\bigl(f_*(a) * b\bigr) = a * f^*(b)
\end{equation}
for $a \in \CH(X)$ and $b \in \CH(Y)$. This relation is Fourier-dual to the usual projection formula.

\begin{proof}
(i) We have a Cartesian diagram
$$
\begin{matrix}
X \times X & \xrightarrow{\qquad\pr_2-\pr_1\qquad} & X\\
\llap{$\scriptstyle f \times \id$}\Big\downarrow && \Big\downarrow\rlap{$\scriptstyle f$}\\
Y \times X & \xrightarrow{\ (\pr_2-\pr_1) \circ (\id \times f)\ } & Y\rlap{\quad .}
\end{matrix}
$$
This gives the relation
$$
(\id \times f)^* (\pr_2-\pr_1)^* f_*(a) = (f \times \id)_* (\pr_2-\pr_1)^*(a)\, .
$$
The LHS defines $f^* \circ \bigl(f_*(a) * - \bigr)$; the RHS defines $(a * -) \circ f^*$.

(ii) As explained above, $(\epsilon * -)$ is given by the class $(\pr_2-\pr_1)^*\bigl(\epsilon\bigr)$ on $(X \times Y)^2$. Using this we easily compute that
$\GG_{\epsilon} \colon h(X)\to h(Y)$ is given by the class $\bigl([-1]_X\times\id_Y\bigr)^*(\eps)$. The first assertion of~(ii) readily follows from this. The second assertion in~(ii) is easily checked using~(i).
\end{proof}

\subsection{{}}
Let $X$ be a $g$-dimensional abelian variety. To prepare for what follows, we explain how to express the Fourier transform $\Fourier_X \colon h(X)_\Q \to h(X^t)_\Q$ using $*$-products. We use the motivic version of the relations given in \S~2 of~\cite{Beauv}. In addition to these, we have, for a homomomorphism $f \colon X \to Y$ with dual $f^t \colon Y^t \to X^t$, the relations
\begin{equation}
\label{FffF}
(f^t)^* \circ \Fourier_X = \Fourier_Y \circ f_*
\quad\text{and}\quad
\Fourier_X \circ f^* = (-1)^{\dim(X)-\dim(Y)} \cdot (f^t)_* \circ \Fourier_Y\, .
\end{equation}
(In \cite{Beauv} these are given only for isogenies.) The first identity in~\eqref{FffF} is easily proven using the relation $(f \times \id)^* \PP_Y \cong (\id \times f^t)^* \PP_X$ on $X \times Y^t$; the second relation follows from the first by Fourier duality.

Let $j_1 \colon \to X \times X^t$ and $j_2 \colon X^t \to X \times X^t$ be the maps given by $x \mapsto (x,0)$ and $\xi \mapsto (0,\xi)$. We  have
\begin{align*}
\FF_X\circ\FF_{X^t} &= \FF_X \circ \pr_{2,*} \circ \bigl(\chern(\PP_{X^t})\cdot -\bigr) \circ \pr_1^*\\
&= j_2^* \circ \FF_{X^t\times X} \circ \bigl(\chern(\PP_{X^t})\cdot -\bigr) \circ \pr_1^*\\
&= j_2^* \circ \bigl(\FF_{X^t\times X}(\chern(\PP_{X^t})) * -\bigr) \circ \FF_{X^t\times X} \circ \pr_1^* \\
&= (-1)^g\cdot j_2^* \circ \bigl(\FF_{X^t\times X}(\chern(\PP_{X^t})) * -\bigr) \circ j_{1,*}\circ \FF_{X^t}\, .
\end{align*}
As $\FF_{X^t}$ is invertible, it follows that
$$
\FF_X = (-1)^g \cdot j_2^* \circ \bigl(\FF_{X^t\times X}(\chern(\PP_{X^t})) * -\bigr) \circ  j_{1,*}\, .
$$

Now assume $X$ has a principal polarization $\lambda \colon X \xrightarrow{\sim} X^t$. We identify $X$ and~$X^t$ via~$\lambda$. Let $\ell := c_1(\PP_X) \in \CH^1(X^2)$, and let $\theta := \frac{1}{2} \cdot (\id \times \lambda)^* \ell$ be the symmetric ample class in $\CH^1(X)_\mQ$ that corresponds to~$\lambda$. We have $\ell = m^*(\theta) - \pr_1^*(\theta) - \pr_2^*(\theta)$.

If $Y$ is an abelian variety and $y$ is a class on~$Y$ such that $\Fourier_Y(y) \in \CH_{>0}(Y)_\Q$ then we have $\Fourier_Y(e^y) = (-1)^{\dim(Y)} \cdot \Exp\bigl((-1)^{\dim(Y)} \cdot \Fourier_Y(y)\bigr)$. Applying this to $Y=X^2$ and $y=\ell$ this gives
$$
\Fourier_{X^2}\bigl(\chern(\PP_X)\bigr) 
= \Exp\bigl(\Fourier_{X^2}(\ell)\bigr)\\
= \Exp\Bigl((-1)^g \cdot \bigl(\Delta_*\Fourier_X(\theta) - j_{1,*}\Fourier_X(\theta) - j_{2,*}\Fourier_X(\theta)\bigr)\Bigr)\, ,
$$
where $j_1$, $j_2 \colon J \to J^2$ are given by $j_1(x) = (x,0)$ and $j_2(x) = (0,x)$. So if we let
\begin{equation}
\label{Elttau}
\tau := (-1)^g \cdot \bigl(\Delta_*\Fourier_X(\theta) - j_{1,*}\Fourier_X(\theta) - j_{2,*}\Fourier_X(\theta)\bigr)
\end{equation}
then the conclusion is that
\begin{equation}
\label{Fourier*Prod}
\Fourier_X = (-1)^g \cdot j_2^* \circ \bigl(\Exp(\tau) * -\bigr) \circ j_{1,*}\, .
\end{equation}
\smallskip

Before we state and prove the main result of this section, we prove a lemma that we need.

\begin{lem}\label{2-torsion-lem} 
Let $A$ be a ring, and let $I\sub A$ be an ideal equipped with a PD-structure. Assume $a\in I$ is an element with $2a=0$. Then
for every $n>0$ we have
$$
2^{1+\lfloor\log_2(n)\rfloor}\cdot a^{[n]}=0\, .
$$
Hence, if $X$ is an abelian variety and $a \in \CH_{>0}(X)$ is a $2$-torsion class then
$$
2^{1+\lfloor\log_2(\dim(X))\rfloor} \cdot \Exp(a)=0\, .$$
\end{lem}

\begin{proof} 
First note that $2^n\cdot a^{[n]}=0$. To see that $a$ is killed by
$2^{1+\lfloor\log_2(n)\rfloor}$ we use induction on~$n$. The case
$n=1$ is clear. Suppose then that $n>1$ and that the statement holds
for all smaller~$n$. If $n$ is not a power of~$2$ then we can write $n=2^rm$ where $m>1$ is odd. It is well known that in this case $\binom{n}{2^r}$ is odd (see e.g.~\cite{Fine}). Indeed,
this follows from the congruences
$$
(1+x)^{2^r m} \equiv (1+x^{2^r})^m\equiv 1+x^{2^r}+\ldots \mod(2)\, .
$$
Now setting $i=2^r$ and using the identity $\binom{n}{i} a^{[n]} = a^{[i]}a^{[n-i]}$ we derive the assertion from the induction assumption. Similarly, if $n$ is a power of~$2$ we use the fact that $\binom{n}{n/2}=2m$, where $m$ is odd.
\end{proof}

We now again consider the situation as in \ref{HECurve-setup}. In $\Pic^{g-1}_{C/k}$ we have the canonical theta divisor~$\Theta$. Let $\theta := t_{(g-1)p_0}^*[\Theta] \in \CH^1(J)$. 
Further, let $c := \bigl[\iota(C)\bigr] \in \CH_1(J)$, and 
define $\gamma \in \CH_1(J^2)$ by 
\begin{equation}
\label{gamma-def}
\gamma := j_{1,*}(c) + j_{2,*}(c) - \Delta_*(c)\, .
\end{equation}

\begin{thm}
\label{HEFourier}
Let $C$ be a hyperelliptic curve of genus $g$ over a field~$k$, and let $J$ be its Jacobian. With notation and assumptions as in~\ref{HECurve-setup} and as above, define a morphism of ungraded integral motives
$$
\Fourier \colon h(J) \to h(J)
$$
by
\begin{equation}
\label{IntFourDef}
\Fourier(a) = (-1)^g \cdot j_2^*\circ\bigl(\Exp(\gamma) * -\bigr)\circ j_{1,*}\, .
\end{equation}
Let $N := 1+\lfloor\log_2(3g)\rfloor$.
Then $\Fourier$ is a motivic integral Fourier transform up to factor $2^N$.
\end{thm}

\begin{proof}
First, we prove condition (i) in Def.~\ref{mot-def}. Pushing forward \eqref{DeltaDelta-} to~$J$ via the morphism $C^2 \to J$ given by $(x,y) \mapsto \iota(x) + \iota(y)$, we find that $[2]_*(c) = 4c$ in $\CH_1(J)$. This means that the image of~$c$ in $\CH_1(J)_\Q$ lies in the subspace on which $[n]_*$ acts multiplication by~$n^2$ for all $n \in \Z$. It follows that $\Fourier_J(c)$ lies in $\CH^1(J)_\Q$. On the other hand, the component of $\Fourier_J(c)$ in codimension~$1$ equals $-\theta$. So the conclusion is that $\theta = -\Fourier_J(c)$; hence $c = (-1)^{g-1} \Fourier_J(\theta)$. {}From this we find that the image of~$\gamma$ in $\CH_1(J^2)_\Q$ equals the element~$\tau$ of~\eqref{Elttau}, and (i) now follows from~\eqref{Fourier*Prod}.

By definition, $\Fourier = (-1)^g \cdot \GG_{\Exp(\gamma)}$.
Thus, to prove~(ii), by Lemma \ref{Pontr-ker-lem}, we have to show that
\begin{equation}
\label{PfiiGoal}
2^{N}\cdot j_{13}^*\bigl(j_{12,*} \Exp(\gamma) * j_{23,*} \Exp(\gamma)\bigr) = 2^{N} \cdot (-1)^g \cdot [\Delta_J]\, ,
\end{equation}
as one easily checks that $\GG_{[\Delta_J]}$ is the operator $[-1]_*$.

Let $\delta \colon J^2 \to J^3$ be the morphism given by $(x,y) \mapsto (x,y,x)$. 
We claim that we have the cubical relation 
\begin{equation}
\label{CubicRelgamma}
2\cdot \bigl(j_{12,*}(\gamma) + j_{23,*}(\gamma)\bigr) = 2\cdot \delta_*(\gamma)
\end{equation}
in~$\CH(J^3)$. Indeed, expanding the definition of~$\gamma$, this relation can be rewritten as
$$
2\cdot \iota^3_*(\De_e) = 0\, ,
$$
which follows from Proposition \ref{cubical-prop}.

Since $\Exp(a+b)=\Exp(a)*\Exp(b)$, Lemma \ref{2-torsion-lem} together with
\eqref{CubicRelgamma} imply that
\begin{equation}
\label{ExpCubicRel}
2^{N} \cdot j_{12,*} \Exp(\gamma) * j_{23,*} \Exp(\gamma) = 2^{N} \cdot \delta_* \Exp(\gamma)\, ,
\end{equation}
which is the key to all further relations we want to prove. As we have a cartesian square
$$
\begin{matrix}
J & \xrightarrow{\ \Delta_J\ } & J^2\\
\llap{$\scriptstyle j_1$}\Big\downarrow && \Big\downarrow\rlap{$\scriptstyle j_{13}$}\\
J^2 &\xrightarrow{\ \delta\ } & J^3
\end{matrix}
$$
we obtain
$$
2^N \cdot j_{13}^*\bigl(j_{12,*} \Exp(\gamma) * j_{23,*} \Exp(\gamma)\bigr) = 2^N \cdot j_{13}^* \delta_* \Exp(\gamma) = 2^N \cdot \Delta_{J,*} j_1^* \Exp(\gamma)\, .
$$

The last step in the proof of (ii) is to show that $2 \cdot j_1^* \Exp(\gamma) = 2 \cdot (-1)^g \cdot [J]$. Expanding $\Exp(\gamma)$ and using \eqref{ProjFormDual} we obtain
\begin{align*}
j_1^* E(\gamma) &= \sum_{m_1, m_2, m_3 \geq 0}\; (-1)^{g-m_3} \cdot j_1^*\bigl(j_{1,*} c^{[g-m_1]} * j_{2,*} c^{[g-m_2]} * \Delta_* c^{[g-m_3]}\bigr)  \\
&= \sum_{m_1, m_2, m_3 \geq 0}\; (-1)^{g-m_3} \cdot c^{[g-m_1]} * j_1^*\bigl(j_{2,*} c^{[g-m_2]} * \Delta_* c^{[g-m_3]}\bigr)\, .
\end{align*}
Next consider the isomorphism $\Phi \colon J^2 \to J^2$ given by $(x,y) \mapsto (x,x-y)$. We have
$$
\Phi \circ j_1 = \Delta\, ,\quad
\Phi \circ j_2 = -j_2\, ,\quad
\Phi \circ \Delta = j_1\, .
$$
Because $[-1]_*(c) = c$ we get, using Lemma~\ref{2WW=2W},
\begin{align*}
2 \cdot j_1^*\bigl(j_{2,*} c^{[g-m_2]} * \Delta_* c^{[g-m_3]}\bigr) &= 2 \cdot \Delta^*\bigl(j_{2,*} c^{[g-m_2]} * j_{1,*} c^{[g-m_3]}\bigr)\\
&= 2 \cdot c^{[g-m_2]} \cdot c^{[g-m_3]} = 2 \cdot \binom{m_2+m_3}{m_2} \cdot c^{[g-m_2-m_3]}\, .
\end{align*}
With this,
$$
2 \cdot j_1^* \Exp(\gamma) = 2 \cdot \sum_{m_1, m_2, m_3 \geq 0}\; (-1)^{g-m_3} \cdot \binom{2g-m_1-m_2-m_3}{g-m_2-m_3} \binom{m_2+m_3}{m_3} \cdot c^{[2g-m_1-m_2-m_3]}\, .
$$
The coefficient of $c^{[2g-M]}$ in this expression is
$$
2 \cdot (-1)^g \cdot \sum_{n\geq 0}\; \binom{2g-M}{g-n} \cdot \left\{\sum_{m_3=0}^n \; \binom{n}{m_3} (-1)^{m_3} \right\} = 2 \cdot (-1)^g \cdot \binom{2g-M}{g}\, .
$$
But only for $2g-M \leq g$ we have a nonzero class $c^{[2g-M]}$; so we obtain the relation $2 \cdot j_1^* \Exp(\gamma) = 2\cdot (-1)^g \cdot [J]$, and the proof of~(ii) is complete.

Finally, let us check the condition (iii) of Def.~\ref{mot-def}, or rather the equivalent
condition~\eqref{switch-F-id}.
We have
\begin{align*}
2^N \cdot (-1)^g \cdot \Fourier\circ\De^*
&= 2^N \cdot j_2^*\circ\bigl(\Exp(\gamma) * -\bigr)\circ j_{1,*}\circ\Delta^*\\
&= 2^N \cdot j_2^*\circ\bigl(\Exp(\gamma) * -\bigr)\circ (\Delta \times \id)^*\circ j_{12,*}\\
&= 2^N \cdot j_2^*\circ (\Delta \times \id)^*\circ \bigl((\Delta \times\id)_* \Exp(\gamma) * -\bigr)\circ 
j_{12,*}\\
&= 2^N \cdot j_3^*\circ \bigl((\Delta \times\id)_* \Exp(\gamma) * -\bigr)\circ j_{12,*}\, ,
\end{align*}
where we make repeated use of~\eqref{ProjFormDual}. On the other hand, using the Cartesian diagram 
$$
\begin{matrix}
J \times J & \xrightarrow{\quad \ m\quad \ } & J\\
\llap{$\scriptstyle j_{34}$}\Big\downarrow && \Big\downarrow\rlap{$\scriptstyle j_3$}\\
J \times J\times J\times J & \xrightarrow{\ \id_{J\times J}\times m\ } & J \times J \times J\\
\end{matrix}
$$
we find
\begin{align*}
m_*\circ(\Fourier\ot\Fourier)
&=m_*\circ j_{34}^*\bigl(j_{13,*}\Exp(\gamma)*j_{24,*}\Exp(\gamma)* -\bigr)\circ j_{12,*}\\
&=j_3^*\circ \bigl((\id_{J\times J} \times m)_*\bigl(j_{13,*}\Exp(\gamma)*j_{24,*}\Exp(\gamma)\bigr)* -\bigr)
\circ (\id_{J\times J}\times m)_*\circ j_{12,*}\\
&=j_3^*\circ \bigl(j_{13,*} \Exp(\gamma) * j_{23,*} \Exp(\gamma) * -\bigr)\circ 
j_{12,*}.
\end{align*}
Thus, the condition \eqref{switch-F-id} is equivalent to the identity
$$
2^N \cdot (\Delta\times \id)_* \Exp(\gamma) = 2^N \cdot j_{13,*} \Exp(\gamma) * j_{23,*} \Exp(\gamma),
$$ 
which is obtained from \eqref{ExpCubicRel} after permutation of the coordinates.
\end{proof}

\begin{cor}
\label{PDIntersProd}
Consider the situation of \ref{HECurve-setup}, and define $\Fourier \colon \CH(J) \to \CH(J)$ 
and the number $N$ as in Theorem \ref{HEFourier}. Then the induced operators
$$
\Fourier \colon 2^N \cdot \CH(J) \to 2^N \cdot \CH(J)
\quad\text{and}\quad
\Fourier \colon \CH(J) \otimes \Z\bigl[1/2\bigr] \to \CH(J) \otimes \Z\bigl[1/2\bigr]
$$
are bijective and satisfy
$$
\Fourier^2 = (-1)^g \cdot [-1]_*\, ,\quad \Fourier(a * b) = \Fourier(a)\cdot \Fourier(b)\, ,\quad \Fourier(a\cdot b) = (-1)^g \cdot \Fourier(a)* \Fourier(b)\, .
$$
If $\ag \subset \CH(J)$ is the augmentation ideal, i.e., the kernel of the projection $\CH(J) \to \CH^0(J) = \Z\cdot [J]$, then we have natural PD-structures $\{\delta_n\}$ on the ideals
$$
2^N \cdot \ag \subset \CH(J)
\quad\text{and}\quad
\ag \otimes \Z\bigl[1/2\bigr] \subset \CH(J) \otimes \Z\bigl[1/2\bigr]
$$
(for the usual intersection product), characterised by the property that $\Fourier\bigl(\delta_n(a)\bigr) = (-1)^{ng} \cdot \Fourier(a)^{[n]}$ for all $a \in 2^N\cdot \ag$ (resp.\ all $a \in \ag \otimes \Z[1/2]$).
\end{cor}

\begin{thm}\label{elliptic-thm} 
Let $E$ be an elliptic curve over an algebraically closed field~$k$. Let $\De_e$ be the modified diagonal cycle in~$E^3$; see Section~\ref{ModDiagCycDef}. If $\phi \colon E \xrightarrow{\sim} E^t$ is an isomorphism, there exists a motivic integral Fourier transform of $(E,\phi)$ up to a factor~$d$ if and only if $d\cdot\De_e=0$. Hence, such a transform exists with $d=2$, but in general not with $d=1$.
\end{thm}

\begin{proof}
Let $\lambda \colon E \xrightarrow{\sim} E^t$ be the polarization given by the origin of~$E$, and write $\phi = \lambda \circ \alpha^{-1}$ for some $\alpha \in \mathrm{Aut}(E)$. Suppose $\Fourier$ is a motivic integral Fourier transform of $(E,\phi)$ up to a factor~$d$. By (ii) of Lemma~\ref{Pontr-ker-lem} we can write $\Fourier = -\GG_\epsilon$ for some class~$\epsilon$ on~$E^2$.

As we have seen in the proof of Thm.~\ref{HEFourier}, the Fourier transform $h(E)_{\Q}\to h(E)_{\Q}$ associated with~$\lambda$ is given by  $-\GG_{\Exp(\gamma)}$, where for an elliptic curve we have 
$$
\gamma = [E \times 0] + [0\times E] - \Delta_E\, .
$$
It follows that the Fourier transform associated with~$\phi$ is given by  $-\GG_{(\id \times \alpha)_* \Exp(\gamma)}$. Hence condition~(i) in Def.~\ref{mot-def} implies that
$\eps$~differs from $(\id\times\a)_* \Exp(\gamma) = \Exp\bigl((\id \times \alpha)_*(\gamma)\bigr)$  
by a torsion class in $\CH(E^2)$. Note that there is no torsion in $\CH_2(E^2)$ and that every torsion class in $\CH_1(E^2)$ is of the form $\pr_1^*(t_1) + \pr_2^*(t_2)$ for some torsion $0$-cycles $t_1$ and~$t_2$ on~$E$. By Prop.~11 of~\cite{Beauv} we can therefore write
$$
\eps = \Exp\bigl((\id\times\a)_*(\gamma)\bigr) + \bigl[(e_1,e_2)\bigr] - \bigl[(0,0)\bigr] + \pr_1^*\bigl([e'_1]-[0]\bigr) + \pr_2^*\bigl([e'_2]-[0]\bigr)
$$
for some torsion points $e_1$, $e_2$, $e'_1$ and~$e'_2$ on~$E$.

The definition of $\gamma$ gives $(\id \times \alpha)_*\bigl(\gamma\bigr) = \bigl[E \times 0\bigr] + \bigl[0\times E\bigr] - \Gamma_\alpha$, and then $\Exp\bigl((\id \times \alpha)_*(\gamma)\bigr) = \bigl[(0,0)\bigr] + (\id \times \alpha)_*\bigl(\gamma\bigr) - [E \times E]$. Hence we obtain
$$
\eps = \bigl[(e_1,e_2)\bigr]+ \bigl[e'_1 \times E\bigr] + \bigl[E \times e'_2\bigr] - \Gamma_\a - [E\times E]\, .
$$
As we have seen in the proof of Thm.~\ref{HEFourier}, conditions (ii) and~(iii) in Def.~\ref{mot-def} are equivalent to the following two relations in $\CH(E^2)$ and
$\CH(E^3)$, respectively:
\begin{align}
&d \cdot j_{13}^*\bigl(j_{12,*}(\eps) * j_{23,*}(\eps)\bigr) = -d \cdot [\De]\, , \label{E-inv-eq}\\
&d \cdot (\De\times\id)_*\bigl(\eps\bigr) = d \cdot j_{13,*}(\eps) * j_{23,*}(\eps)\, . \label{E-prod-eq}
\end{align}

If in~\eqref{E-prod-eq} we take components in dimension~$0$ and push forward via $\pr_3\colon E^3\to E$ we obtain the relation $d\cdot e_2=0$. Next, pulling back the dimension~$2$ component of \eqref{E-prod-eq} by~$j_1$ we find the relation
$$
d \cdot \Bigl\{[0] - [e_1] + [0] - \bigl[-\alpha^{-1}(e_2^\prime)\bigr] \Bigr\} = 0\, .
$$
This implies that $d\cdot \a(e_1) = d\cdot e'_2$ in~$E$. On the other hand, \eqref{E-inv-eq} gives in dimension~$0$ the relation
$$
d\cdot \Bigl\{ \bigl[(e_1,e'_2)\bigr] - \bigl[(e_1,0)\bigr] + \bigl[(e'_1,0)\bigr] - \bigl[(-\a^{-1}(e_1),0)\bigr]\Bigr\} = 0\, ,
$$
which implies that $d\cdot e'_2=0$ and $d\cdot \a(e'_1) = d\cdot e_1$. Combining this with the previous relations we derive that $e_1$, $e_2$, $e'_1$ and~$e'_2$ are all $d$-torsion points. Therefore, the validity of \eqref{E-inv-eq} and~\eqref{E-prod-eq} does not change if we replace them by zero, i.e., if we replace~$\eps$ by $\Exp\bigl((\id\times\a)_*(\gamma)\bigr)$. But then the  component of~\eqref{E-prod-eq} in dimension~$1$ gives $d\cdot (\id_{E^2} \times \a)_*\bigl(\De_e\bigr) = 0$, or simply $d\cdot\De_e=0$. 

Conversely, if $d\cdot\De_e=0$ then we claim that \eqref{ExpCubicRel} holds with $2^N$ replaced by~$d$, i.e.,
\begin{equation}\label{ExpCubicBIS}
d \cdot j_{12,*} \Exp(\gamma) * j_{23,*} \Exp(\gamma) = d \cdot \delta_* \Exp(\gamma)\, .
\end{equation}
For the components in dimension~$0$ this relation is clear. For the $1$-dimensional components we just have the relation $d \cdot \Delta_e = 0$ that we are assuming. To verify~\eqref{ExpCubicBIS} in dimension~$2$ (i.e., codimension~$1$), note that
$$
\Exp(\gamma) = \bigl[(0,0)\bigr] + \gamma - [E \times E] = \bigl[(0,0)\bigr] + [E \times 0] + [0\times E] - \Delta_E - [E \times E] \, .
$$
The relation in codimension~$1$ then follows (even without a factor~$d$) by direct calculation, using the theorem of the cube. For the components in dimension~$3$, finally, \eqref{ExpCubicBIS} follows from the remark that 
$$
j_{12,*}(\gamma) * j_{23,*}[E \times E] = 0 = j_{12,*}[E\times E] * j_{23,*}(\gamma)\, .
$$

Once we have relation~\eqref{ExpCubicBIS}, we can copy the proof of Thm.~\ref{HEFourier}, everywhere replacing $2^N$ by~$d$. The conclusion is that the~$\Fourier$ of~\eqref{IntFourDef}, now with $C=J=E$, is an integral motivic Fourier transform of $(E,\phi)$ up to a factor~$d$, for any $\phi \colon E \xrightarrow{\sim} E^t$. 

For the last assertion of the theorem, recall from Prop.~\ref{cubical-prop} that $2 \cdot \Delta_e = 0$, whereas by Remark~\ref{cubical-rem} in general $\Delta_e \neq 0$.
\end{proof}


\begin{thebibliography}{99}
\bibitem{Beauv}
A.~Beauville, Quelques remarques sur la transformation de Fourier dans l'anneau de Chow d'une vari\'et\'e ab\'elienne. In: Algebraic geometry (Tokyo/Kyoto, 1982), 238--260. Lecture Notes in Math., 1016, Springer, Berlin, 1983.

\bibitem{BLR}
S.~Bosch, W.~L\"utkebohmert and M.~Raynaud, N\'eron models. Ergebnisse der Mathematik und ihrer Grenzgebiete (3), 21. Springer-Verlag, Berlin, 1990.

\bibitem{BST}
J.~Buhler, C.~Schoen and J.~Top,
Cycles, $L$-functions and triple products of elliptic curves.
J. Reine Angew. Math. 492 (1997), 93--133. 

\bibitem{Collino}
A.~Collino, Poincar\'e's formulas and hyperelliptic
curves, Atti Accad. Sci. Torino Cl. Sci. Fis. Mat. Natur. 109 (1975), no. 1-2, 89--101.

\bibitem{SGA4-17}
P.~Deligne, Cohomologie \`a supports propres. In: Th\'eorie des topos et cohomologie \'etale des sch\'emas, tome~3. (SGA4, Exp.\ XVII.) Lecture Notes Math.\ 305, Springer-Verlag, Berlin, 1973.

\bibitem{DG}
M.~Demazure and P.~Gabriel, Groupes Alg\'ebriques, Tome I. Masson \& Cie, Paris/North-Holland, Amsterdam, 1970.

\bibitem{Ekedahl}
T.~Ekedahl, The action of monodromy on torsion points of Jacobians. In: Arithmetic algebraic geometry (Texel, 1989), Progr. Math. 89, Birkh\"auser, Boston, MA, 1991, pp.\  41--49.

\bibitem{Esnault}
H.~Esnault, Some elementary theorems about divisibility of $0$-cycles on abelian varieties defined over finite fields. Int. Math. Res. Not. 2004, no. 19, 929--935. 


\bibitem{Fine} 
N.J.~Fine, Binomial coefficients modulo a prime, Amer. Math. Monthly 54 (1947), 589--592.

\bibitem{Fulton} 
W.~Fulton, Intersection theory, 2nd edition. Ergebnisse der Mathematik und ihrer Grenzgebiete, 3. Folge, Vol.~2. Springer-Verlag, Berlin, 1998.

\bibitem{GS} B.~Gross and C. Schoen, The modified diagonal cycle on the triple product of a pointed curve, Annales de l'institut Fourier 45 (1995), 649--679.

\bibitem{Matsumura}
H.~Matsumura, Commutative ring theory. Cambridge Studies in Advanced Mathematics, 8. Cambridge University Press, Cambridge, 1986.

\bibitem{Milne}
J.S.~Milne, Zero cycles on algebraic varieties in nonzero characteristic: Rojtman's theorem. Compositio Math. 47 (1982), no. 3, 271--287.

\bibitem{PartII}
B.~Moonen and A.~Polishchuk, Algebraic cycles on the relative symmetric powers and on the relative Jacobian of a family of curves. II, arXiv:0805.3621.

\bibitem{Mumford} 
D.~Mumford, Abelian varieties. Tata Institute of Fundamental Research Studies in Math. 5. Oxford University Press, Oxford, 1970.

\bibitem{P-sym} 
A.~Polishchuk, Algebraic cycles on the relative symmetric powers and on the relative Jacobian of a family of curves. I. Selecta Math. (N.S.) 13 (2007), 531--569.

\bibitem{Roitman}
A.A.~Rojtman, The torsion of the group of $0$-cycles modulo rational equivalence.
Ann. of Math. (2) 111 (1980), no. 3, 553--569. 


\end{thebibliography}
\end{document}